\documentclass{article}
\usepackage{amsmath}
\usepackage{amssymb}
\usepackage{amsfonts}

\newtheorem{theorem}{Theorem}

\newtheorem{lemma}[theorem]{Lemma}

\newtheorem{proposition}[theorem]{Proposition}

\newenvironment{proof}[1][Proof]{\noindent\textbf{#1.} }{\ \rule{0.5em}{0.5em}}
\input{tcilatex}
\begin{document}

\title{Pierce stalks in preprimal varieties}
\author{D. Vaggione and W. J. Zuluaga Botero}
\maketitle

\begin{abstract}
An algebra $\mathbf{P}$ is called \textit{preprimal} if $\mathbf{P}$ is
finite and $\func{Clo}(\mathbf{P})$ is a maximal clone. A \textit{preprimal
variety} is a variety generated by a preprimal algebra. After Rosenberg's
classification of maximal clones \cite{ro}; we have that a finite algebra is
preprimal if and only if its term operations are exactly the functions
preserving a relation of one of the following seven types:

\begin{enumerate}
\item Permutations with cycles all the same prime length,

\item Proper subsets,

\item Prime-affine relations,

\item Bounded partial orders,

\item $h$-adic relations,

\item Central relations $h\geq 2$,

\item Proper, non-trivial equivalence relations.
\end{enumerate}

In \cite{kn} Knoebel studies the Pierce sheaf of the different preprimal
varieties and he asks for a description of the Pierce stalks. He solves this
problem for the cases 1.,2. and 3. and left open the remaining cases. In
this paper, using central element theory we succeeded in describing the
Pierce stalks of the cases 6. and 7..
\end{abstract}

\section{Introduction}

An algebra $\mathbf{P}$ is called \textit{preprimal} if $\mathbf{P}$ is
finite and $\func{Clo}(\mathbf{P})$ is a maximal clone. A \textit{preprimal
variety} is a variety generated by a preprimal algebra.

After Rosenberg's classification of maximal clones \cite{ro}; we have that a
finite algebra is preprimal if and only if its term operations are exactly
the functions preserving a relation of one of the following seven types:

\begin{enumerate}
\item Permutations with cycles all the same prime length,

\item Proper subsets,

\item Prime-affine relations,

\item Bounded partial orders,

\item $h$-adic relations,

\item Central relations $h\geq 2$,

\item Proper, non-trivial equivalence relations.
\end{enumerate}

In \cite{kn} Knoebel studies the Pierce sheaf (\cite{pi}, \cite{co}) of the
different preprimal varieties and he asks for a description of the Pierce
stalks. He solves this problem for the cases 1.,2. and 3. and left open the
remaining cases.

In this paper, using central element theory we succeeded in describing the
Pierce stalks of the cases 6. and 7..

\bigskip

\section{Notation and some basic results}

Let $\{\mathbf{A}_{i}\}_{i\in I}$ be a family of algebras of the same type.
If there is no place to confusion, we write $\prod \mathbf{A}_{i}$ in place
of $\prod_{i\in I}\mathbf{A}_{i}$. For $x,y\in \prod A_{i}$, let $%
E(x,y)=\left\{ i\in I:x(i)=y(i)\right\} $. A subdirect product $\mathbf{A}%
\leq \prod \mathbf{A}_{i}$ is \textit{global} \cite{kr-cl} if there is a
topology $\tau $ on $I$ such that $E(x,y)\in \tau $, for every $x,y\in A$,
and

\begin{enumerate}
\item[-] (Patchwork Property) For every $\{U_{r}:r\in R\}\subseteq \tau $
and $\{x_{r}:r\in R\}\subseteq A$, if $\bigcup \{U_{r}:r\in R\}=I$ and $%
U_{r}\cap U_{s}\subseteq E(x_{r},x_{s})$, for every $r,s\in R$, then there
exists $x\in A$ such that $U_{r}\subseteq E(x,x_{r})$, for every $r\in R$.
\end{enumerate}

\noindent Global subdirect products were introduced in \cite{kr-cl} as a
universal algebra counterpart to sheaves, where it is proved that a global
subdirect product and an algebra of global sections of a sheaf are the same
thing.

For an algebra $\mathbf{A}$\ we use $\nabla ^{\mathbf{A}}$ to denote the
universal congruence on $\mathbf{A}$ and $\Delta ^{\mathbf{A}}$ to denote
the trivial congruence on $\mathbf{A}$. We use $\mathrm{Con}(\mathbf{A})$ to
denote the congruence lattice of $\mathbf{A}$. We use $\theta ^{\mathbf{A}%
}(a,b)$ to denote the principal congruence of $\mathbf{A}$ generated by $%
(a,b)$. If $\vec{a},\vec{b}\in A^{n}$, we use $\theta ^{\mathbf{A}}(\vec{a},%
\vec{b})$ to denote the congruence $\bigvee_{k=1}^{n}\theta ^{\mathbf{A}%
}(a_{k},b_{k})$.

\begin{lemma}
\label{homomorfismos preservan}If $\sigma :\mathbf{A}\rightarrow \mathbf{B}$
is a homomorphism and $(x,y)\in \theta ^{\mathbf{A}}(\vec{a},\vec{b})$, then 
$(\sigma (x),\sigma (y))\in \theta ^{\mathbf{B}}(\sigma (\vec{a}),\sigma (%
\vec{b}))$.
\end{lemma}

\begin{proof}
Folklore.
\end{proof}

\bigskip 

If $\mathbf{A}\leq \prod \mathbf{A}_{i}$, for each $i\in I$, let $\pi _{i}^{%
\mathbf{A}}:\mathbf{A}\rightarrow \mathbf{A}_{i}$ be the canonical
projection. We use $\theta _{i}^{\mathbf{A}}$ to denote the congruence $\ker
\pi _{i}^{\mathbf{A}}$.

Given a class $\mathcal{K}$ of algebras, we use $\mathbb{I}(\mathcal{K})$, $%
\mathbb{H}(\mathcal{K})$, $\mathbb{S}(\mathcal{K})$, $\mathbb{P}_{u}(%
\mathcal{K})$ and $\mathbb{V}(\mathcal{K})$ to denote the classes of
isomorphic images, homomorphic images, subalgebras, ultraproducts of
elements of $\mathcal{K}$ and the variety generated by $\mathcal{K}$. For a
variety $\mathcal{V}$, we write $\mathcal{V}_{SI}$ and $\mathcal{V}_{DI}$ to
denote the classes of subdirectly irreducible and directly indecomposable
members of $\mathcal{V}$.

Given a variety $\mathcal{V}$\ and a set $X$ of variables we use $\mathbf{F}%
_{\mathcal{V}}(X)$ for the free algebra of $\mathcal{V}$ freely generated by 
$X$.

\begin{lemma}
\label{technical lemma}Let $\mathcal{V}$ be a variety and let $p_{k}(\vec{z},%
\vec{w}),q_{k}(\vec{z},\vec{w})$, $k=1,...,n$, be terms in the language of $%
\mathcal{V}$. Let $X=\{x,y,\vec{z},\vec{w}\}$,%
\begin{equation*}
\theta =\bigvee_{k=1}^{n}\theta ^{\mathbf{F}_{\mathcal{V}}(X)}(p_{k}^{%
\mathbf{F}_{\mathcal{V}}(X)}(\vec{z},\vec{w}),q_{k}^{\mathbf{F}_{\mathcal{V}%
}(X)}(\vec{z},\vec{w}))
\end{equation*}%
and $\mathbf{H}=\mathbf{F}_{\mathcal{V}}(X)/\theta $. Let $\mathbf{A}\in 
\mathcal{V}$ and suppose that $p_{k}^{\mathbf{A}}(\vec{e},\vec{a})=q_{k}^{%
\mathbf{A}}(\vec{e},\vec{a})$, for $k=1,...,n$. Then, for every $a,b\in A$,
there exists a unique homomorphism $\Omega :\mathbf{H}\rightarrow \mathbf{A}$%
, such that $\Omega (x/\theta )=a$, $\Omega (y/\theta )=b$, $\Omega (\vec{z}%
/\theta )=\vec{e}$ and $\Omega (\vec{w}/\theta )=\vec{a}$.
\end{lemma}

\begin{proof}
The proof is straightforward. The details are left to the reader.
\end{proof}

\bigskip 

A variety $\mathcal{V}$ has the \emph{Fraser-Horn property} (\emph{FHP}, for
short) \cite{FH1970} if for every $\mathbf{A}_{1},\mathbf{A}_{2}\in \mathcal{%
V}$, it is the case that every congruence $\theta $ in $\mathbf{A}_{1}\times 
\mathbf{A}_{2}$ is the product congruence $\theta _{1}\times \theta _{2}$
for some congruences $\theta _{1}$ of $\mathbf{A}_{1}$ and $\theta _{2}$ of $%
\mathbf{A}_{2}$.

Let $\mathcal{L}$ be a first order language. If a $\mathcal{L}$-formula $%
\varphi (\vec{x})$ has the form 
\begin{equation*}
{\exists }{\vec{w}}\bigwedge_{j=1}^{n}p_{j}(\vec{x},\vec{w})=q_{j}(\vec{x},%
\vec{w}),
\end{equation*}%
for some positive number $n$ and terms $p_{j}(\vec{x},\vec{w})$ and $q_{j}(%
\vec{x},\vec{w})$ in $\mathcal{L}$, then we say that $\varphi (\vec{x})$ is
a ($\exists \bigwedge p=q$)-\emph{formula}. In a similar manner we define ($%
\forall \bigwedge p=q$) and ($\forall \exists \bigwedge p=q$)-formulas.

Let $\mathcal{L}$ be a first order language and $\mathcal{K}$ be a class of $%
\mathcal{L}$-structures. If $R\in \mathcal{L}$ is an $n$-ary relation
symbol, we say that a formula $\varphi (x_{1},...,x_{n})$ \emph{defines} $R$ 
\emph{in} $\mathcal{K}$ if%
\begin{equation*}
\mathcal{K}\vDash \varphi (\vec{x})\leftrightarrow R(\vec{x})\text{.}
\end{equation*}
If $\mathcal{L}\subseteq \mathcal{L}^{\prime }$ are first order languages,
then for a $\mathcal{L}^{\prime }$-structure $\mathbf{A}$, we use $\mathbf{A}%
_{\mathcal{L}}$ to denote the reduct of $\mathbf{A}$ to the language $%
\mathcal{L}$. If $\vec{a}_{i}=(a_{i1},...,a_{in})\in A_{i}^{n}$ with $1\leq
i\leq m$, then we write $[\vec{a}_{1},...,\vec{a}_{m}]$ to denote the $n$%
-tuple $((a_{11},...,a_{m1}),...,(a_{1n},...,a_{mn}))\in (\prod A_{i})^{n}$.

\begin{lemma}
\label{lema semantico} Let $\mathcal{L}$ be a language of algebras, $R$ be a 
$n$-ary relation symbol and consider $\mathcal{L}^{\prime }=\mathcal{L}\cup
\{R\}$. Let $\mathcal{K}$ be any class of $\mathcal{L}^{\prime }$%
-structures. The following are equivalent:

\begin{enumerate}
\item There is an existential positive $\mathcal{L}$-formula which defines $%
R $ in $\mathcal{K}$.

\item For all $\mathbf{A},\mathbf{B}\in \mathbb{P}_{u}(\mathcal{K})$ and all
homomorphisms $\sigma :\mathbf{A}_{\mathcal{L}}\rightarrow \mathbf{B}_{%
\mathcal{L}}$, we have that $\sigma :\mathbf{A}\rightarrow \mathbf{B}$ is a
homomorphism.
\end{enumerate}

Moreover, if $\mathcal{K}$ is closed under the formation of finite direct
products, then $1.$ and $2.$ are equivalent to

\begin{itemize}
\item[3.] $R$ is definable in $\mathcal{K}$ by a $(\exists \bigwedge p=q)$%
-formula.
\end{itemize}
\end{lemma}

\begin{proof}
The proof of the equivalence of $1.$ and $2.$ can be obtained from $(3)$ of
Theorem 6.2 in \cite{ca-va3}. Of course $3.$ implies $1.$. We will show that 
$1.$ implies $3.$. Suppose%
\begin{equation*}
\phi (\vec{z})=\exists \vec{w}\bigvee_{j=1}^{m}\bigwedge_{k=1}^{n_{j}}p_{jk}(%
\vec{z},\vec{w})=q_{jk}(\vec{z},\vec{w})
\end{equation*}%
defines $R$ in $\mathcal{K}$. We will prove that for some $j\in \{1,...,m\}$%
, the formula $\psi _{j}(\vec{z})=\exists \vec{w}%
\bigwedge_{k=1}^{n_{j}}p_{k}(\vec{z},\vec{w})=q_{k}(\vec{z},\vec{w})$
defines $R$ in $\mathcal{K}$. We proceed by contradiction. So suppose that
no formula $\psi _{j}(\vec{z})$ defines $R$ in $\mathcal{K}$. Then for every 
$j$ there exists $\mathbf{A}_{j}\in \mathcal{K}$ and $\vec{e}_{j}\in R^{%
\mathbf{A}_{j}}$, such that $\mathbf{A}_{j}\models \lnot \psi _{j}(\vec{e}%
_{j})$. Let $\vec{e}=[\vec{e}_{1},...,\vec{e}_{m}]$. Since $\vec{e}\in
R^{\Pi \mathbf{A}_{j}}$ and $\prod \mathbf{A}_{j}\in \mathcal{K}$, we have
that $\prod \mathbf{A}_{j}\models \phi (\vec{e})$. Naturally, this says that 
$\prod \mathbf{A}_{j}\models \psi _{j}(\vec{e})$, for some $j$ and hence $%
\mathbf{A}_{j}\models \psi _{j}(\vec{e}_{j})$, for some $j$, which produces
a contradiction.
\end{proof}

\bigskip

\section{Central elements}

In a universal-algebraic setting, one key concept for the study of the
Pierce sheaf is that of central element. This tool can be developed
fruitfully in varieties with $\vec{0}$ and $\vec{1}$, which we now define.

A \textit{variety with} $\vec{0}$ \textit{and} $\vec{1}$ is a variety $%
\mathcal{V}$ in which there are 0-ary terms $0_{1},\ldots ,0_{N},$ $%
1_{1},\ldots ,1_{N}$ such that $\mathcal{V}\vDash \vec{0}=\vec{1}\rightarrow
x=y$, where $\vec{0}=(0_{1},\ldots ,0_{N})$ and $\vec{1}=(1_{1},\ldots
,1_{N})$. The terms $\vec{0}$ and $\vec{1}$ are analogue, in a rather
general manner, to identity (top) and null (bottom) elements in rings
(lattices), and its existence in a variety, when the language has at least a
constant symbol, is equivalent to the fact that no non-trivial algebra in
the variety has a trivial subalgebra (see Proposition 2.3 of \cite{ca-va0}).

If $\mathbf{A}\in \mathcal{V}$, then we say that $\vec{e}\in A^{N}$ is a 
\textit{central element} of $\mathbf{A}$ if there exists an isomorphism $%
\mathbf{A}\rightarrow \mathbf{A}_{1}\times \mathbf{A}_{2}$ such that $\vec{e}%
\rightarrow \lbrack \vec{0},\vec{1}]$. We use $Z(\mathbf{A})$ to denote the
set of central elements of $\mathbf{A}$.

Central elements are a generalization of both central idempotent elements in
rings with identity and neutral complemented elements in bounded lattices.
In these classical cases it is well known that the central elements
concentrate the information concerning the direct product representations.
This happens when $\mathcal{V}$ has the Fraser-Horn property \cite{va7}%
\footnote{%
In \cite{sate-va} it is solved the problem of characterizing those varieties
with $\vec{0}$ and $\vec{1}$ in which central elements determine the direct
product representations. These are the varieties with definable factor
congruences or equivalently, the varieties with Boolean factor congruences.
See \cite{ba-va} for a non constructive short proof of this fact.}. It is
well known that the set of factor congruences of an algebra $\mathbf{A}$ in
a variety with the Fraser-Horn property forms a Boolean algebra $FC(\mathbf{A%
})$ which is a sublattice of $\mathrm{Con}(\mathbf{A})$ (see \cite{bi-bu}).
In \cite{va7} it is proved that if $\mathcal{V}$ has the Fraser-Horn
property, then for $\mathbf{A}\in \mathcal{V}$, the map%
\begin{equation*}
\begin{array}{rcl}
\lambda :FC(\mathbf{A}) & \rightarrow  & Z(\mathbf{A}) \\ 
\theta  & \rightarrow  & 
\begin{array}[t]{l}
\mathrm{unique\ }\vec{e}\in A^{N}\ \mathrm{such\ that} \\ 
\mathrm{\ \ \ \ \ \ \ \ \ }\vec{e}\equiv \vec{0}(\theta )\ \mathrm{and\ }%
\vec{e}\equiv \vec{1}(\theta ^{\ast })\QQfnmark{%
We write $\vec{a}\equiv \vec{b}(\theta )$ to express that $(a_{i},b_{i})\in
\theta ,$ $i=1,...,N$.}%
\end{array}%
\end{array}%
\QQfntext{0}{
We write $\vec{a}\equiv \vec{b}(\theta )$ to express that $(a_{i},b_{i})\in
\theta ,$ $i=1,...,N$.}
\end{equation*}%
(where $\theta ^{\ast }$ is the complement of $\theta $ in $FC(\mathbf{A})$)
is bijective. Thus via the above bijection we can give to $Z(\mathbf{A})$ a
Boolean algebra structure. We shall denote by $\mathbf{Z}(\mathbf{A})$ this
Boolean algebra. Many of the usual properties of central elements in rings
with identity or bounded lattices hold when $\mathcal{V}$ has the FHP. We
say that a set of first order formulas $\{\varphi _{r}(\vec{z}):r\in R\}$ 
\textit{defines the property} $``\vec{e}\in Z(\mathbf{A})"$ \textit{in} $%
\mathcal{V}$ if for every $\mathbf{A}\in \mathcal{K}$ and $\vec{e}\in A^{N}$%
, we have that $\vec{e}\in Z(\mathbf{A})$ iff $\mathbf{A}\vDash \varphi _{r}(%
\vec{e})$, for every $r\in R$.

In \cite{va7} it is proved

\begin{lemma}
\label{Lema FHP} Let $\mathcal{V}$ be a variety with $\vec{0}$ and $\vec{1}$
with the FHP. Let $\mathcal{L}$ be the language of $\mathcal{V}$. Then,
there is a finite set $\Sigma _{0}$ of $(\forall \exists \bigwedge p=q)$%
-formulas in the variables $z_{1},...,z_{N}$ and $(\forall \bigwedge p=q)$%
-formulas%
\begin{equation*}
F\text{-}PRES(z_{1},...,z_{N})\text{, with }F\in \mathcal{L}
\end{equation*}%
such that $\Sigma _{0}\cup \{F$-$PRES(\vec{z}):F\in \mathcal{L}\}$ defines
the property $``\vec{e}\in Z(\mathbf{A})"$.
\end{lemma}

\bigskip

Also in \cite{va7} it is proved that there is a $(\exists \bigwedge p=q)$%
-formula $\varepsilon (x,y,\vec{z})$ such that for all $\mathbf{A},\mathbf{B}%
\in $ $\mathcal{V}$,%
\begin{equation*}
\mathbf{A}\times \mathbf{B}\vDash \varepsilon ((a,b),(c,d),[\vec{0},\vec{1}])%
\text{ if and only if }a=c\text{.}
\end{equation*}%
The formula $\varepsilon (\_,\_,\vec{e})$ defines the factor congruence
associated (via the map $\lambda ^{-1}$) with the central element $\vec{e}$.
We stress that the existence of $\varepsilon (x,y,\vec{z})$ and the set $%
\Sigma _{0}\cup \{F$-$PRES(\vec{z}):F\in \mathcal{L}\}$ imply that the
central elements (and its Boolean algebra structure) are preserved by
surjective homomorphisms and products. That is to say:

\begin{lemma}
\label{In FHP surjections give homomorphisms} Let $\mathcal{V}$ be a variety
with $\vec{0}$ and $\vec{1}$ with the FHP and let $\mathbf{A},\mathbf{B}\in 
\mathcal{V}$. If $f:\mathbf{A}\rightarrow \mathbf{B}$ is a surjective
homomorphism, then the map $Z(\mathbf{A})\rightarrow Z(\mathbf{B})$ defined
by $\vec{e}\mapsto (f(e_{1}),...,f(e_{N}))$, is a homomorphism from $\mathbf{%
Z}(\mathbf{A})$ to $\mathbf{Z}(\mathbf{B})$.
\end{lemma}

\begin{lemma}
\label{Products Center FHP} Let $\mathcal{V}$ be a variety with $\vec{0}$
and $\vec{1}$ with the FHP, $\{\mathbf{A}_{i}\}_{i\in I}$ be a family of
members of $\mathcal{V}$ and $\vec{e} \in (\prod \mathbf{A}_{i})^{N}$. Then, 
$\vec{e}\in Z(\prod \mathbf{A}_{i})$ if and only if $\vec{e}(i)\in Z(\mathbf{%
A}_{i})$ for every $i\in I$. Moreover, $\mathbf{Z}(\prod \mathbf{A}_{i})$ is
naturally isomorphic to $\prod \mathbf{Z}(\mathbf{A}_{i})$.
\end{lemma}

\bigskip 

For the details of the proofs of Lemmas \ref{In FHP surjections give
homomorphisms} and \ref{Products Center FHP}, the reader may consult the
proofs of the items $(a)$ and $(b)$ of Lemma 4 in \cite{va4}.

\bigskip

\subsection{Key theorem}

Next, we will prove a series of lemmas in order to demonstrate a result
(Theorem \ref{exitencia del u}) which will be fundamental in our study of
the Pierce stalks.

\begin{lemma}
\label{centrales-complementarios}Let $\mathcal{V}$ be a variety with $\vec{0}
$ and $\vec{1}$ with the FHP such that $\mathbb{P}_{u}(\mathcal{V}%
_{SI})\subseteq \mathcal{V}_{DI}$. Then, the property $``\vec{e}\in Z(%
\mathbf{A})"$ is definable in $\mathcal{V}$ by a single first order formula.
\end{lemma}

\begin{proof}
By Lemma \ref{Lema FHP}, there is a finite set $\Sigma _{0}$ of $(\forall
\exists \bigwedge p=q)$-formulas in the variables $z_{1},...,z_{N}$ and $%
(\forall \bigwedge p=q)$-formulas%
\begin{equation*}
F\text{-}PRES(z_{1},...,z_{N})\text{, with }F\in \mathcal{L}
\end{equation*}%
such that $\Sigma _{0}\cup \{F$-$PRES(\vec{z}):F\in \mathcal{L}\}$ defines
the property $``\vec{e}\in Z(\mathbf{A})"$. Since $\mathbb{P}_{u}(\mathcal{V}%
_{SI})\subseteq \mathcal{V}_{DI}$ we have that%
\begin{equation*}
\mathbb{P}_{u}(\mathcal{V}_{SI})\models \left( \bigwedge_{\varphi \in \Sigma
_{0}}\varphi (\vec{z})\wedge \bigwedge_{F\in \mathcal{L}}F\text{-}PRES(\vec{z%
})\right) \rightarrow \left( (\vec{z}=\vec{0})\vee (\vec{z}=\vec{1})\right)
\end{equation*}%
So, by compactness there exists a finite subset $\mathcal{L}_{0}\subseteq 
\mathcal{L}$ such that:%
\begin{equation}
\mathbb{P}_{u}(\mathcal{V}_{SI})\models \left( \bigwedge_{\varphi \in \Sigma
_{0}}\varphi (\vec{z})\wedge \bigwedge_{F\in \mathcal{L}_{0}}F\text{-}PRES(%
\vec{z})\right) \rightarrow \left( (\vec{z}=\vec{0})\vee (\vec{z}=\vec{1}%
)\right)  \label{equation ast}
\end{equation}%
We will see that the formula $\psi(\vec{z}) =\bigwedge_{\varphi \in \Sigma
_{0}}\varphi (\vec{z})\wedge \bigwedge_{F\in \mathcal{L}_{0}}F$-$PRES(\vec{z}%
)$ defines the property $``\vec{e}\in Z(\mathbf{A})"$ in $\mathcal{V}$. To
do so, it is enough to check that if $\mathbf{A}\in \mathcal{V}$ and $%
\mathbf{A}\models \psi (\vec{e})$, then $\mathbf{A}\models F\text{-}PRES(%
\vec{e})$, for every $F\in \mathcal{L}-\mathcal{L}_{0}$. From Birkhoff's
subdirect representation theorem, we can assume that $\mathbf{A}\leq \prod 
\mathbf{A}_{i}$ is a subdirect product with subdirectly irreducible factors.
Since the formula $\psi $ is positive and the $\mathbf{A}_{i}$ are quotients
of $\mathbf{A}$, we obtain that $\mathbf{A}_{i}\models \psi (\vec{e}(i))$,
for every $i\in I$. But every $\mathbf{A}_{i}$ belongs to $\mathbb{P}_{u}(%
\mathcal{V}_{SI})$ so, from (\ref{equation ast}) we have that $\vec{e}(i)\in
\{\vec{0}^{\mathbf{A}_{i}},\vec{1}^{\mathbf{A}_{i}}\}$, for every $i\in I$.
Since $\vec{0}^{\mathbf{A}_{i}},\vec{1}^{\mathbf{A}_{i}}\in Z(\mathbf{A}%
_{i}) $ we have that 
\begin{equation*}
\mathbf{A}_{i}\models F\text{-}PRES(\vec{0}^{\mathbf{A}_{i}})\wedge F\text{-}%
PRES(\vec{1}^{\mathbf{A}_{i}})\text{, for every }F\in \mathcal{L}
\end{equation*}%
therefore, we get that 
\begin{equation*}
\mathbf{A}_{i}\models F\text{-}PRES(\vec{e}(i))\text{, for every }F\in 
\mathcal{L}-\mathcal{L}_{0}
\end{equation*}%
for every $i\in I$. But the formulas $F\text{-}PRES(z_{1},...,z_{N})$ are $%
(\forall \bigwedge p=q)$-formulas, so they are preserved by subdirect
products. Therefore, we can conclude that 
\begin{equation*}
\mathbf{A}\models F\text{-}PRES(\vec{e})\text{, for every }F\in \mathcal{L}-%
\mathcal{L}_{0}.
\end{equation*}
\end{proof}

\begin{lemma}
\label{Preservation centrals} Let $\mathcal{V}$ be a variety with $\vec{0}$
and $\vec{1}$ with the FHP, such that $\mathbb{P}_{u}(\mathcal{V}%
_{SI})\subseteq \mathcal{V}_{DI}$. Then, the following are equivalent:

\begin{enumerate}
\item The property $``\vec{e}\in Z(\mathbf{A})"$ is definable in $\mathcal{V}
$ by an existential positive formula.

\item The homomorphisms in $\mathcal{V}$ preserve central elements. I.e. if $%
\sigma :\mathbf{A}\rightarrow \mathbf{B}$ is a homomorphism between elements
of $\mathcal{V}$ and $\vec{e}\in Z(\mathbf{A})$, then $\sigma (\vec{e})\in Z(%
\mathbf{B}) $.

\item The property $``\vec{e}\in Z(\mathbf{A})"$ is definable in $\mathcal{V}
$ by a $(\exists \bigwedge p=q)$-formula.
\end{enumerate}
\end{lemma}

\begin{proof}
Let $\mathcal{L}$ be the language of $\mathcal{V}$ and let $\mathcal{L}%
^{\prime }=\mathcal{L}\cup \{R\}$ where $R$ is a $N$-ary relation symbol.
Given an algebra $\mathbf{A}\in \mathcal{V}$ we define:%
\begin{equation*}
R^{\mathbf{A}}=Z(\mathbf{A})
\end{equation*}%
Let $\mathcal{K}$ be the following class of $\mathcal{L}^{\prime }$%
-structures%
\begin{equation*}
\mathcal{K}=\{(\mathbf{A},R^{\mathbf{A}}):\mathbf{A}\in \mathcal{V}\}
\end{equation*}
From Lemma \ref{centrales-complementarios}, there exists a first order
formula $\sigma (z_{1},...,z_{N})$, in the language of $\mathcal{V}$, such
that for every $\mathbf{A}\in \mathcal{V}$, we have that $\vec{e}\in Z(%
\mathbf{A})$ if and only if $\mathbf{A}\models \sigma (\vec{e})$. That is to
say, for every $\mathbf{A}\in \mathcal{V}$, we have that $%
(e_{1},...,e_{N})\in R^{\mathbf{A}}$ if and only if $\mathbf{A}\models
\sigma (\vec{e})$. It is easy to see that the class $\mathcal{K}$ is
axiomatizable by the set of sentences%
\begin{equation*}
\Sigma \cup \{\forall z_{1}...z_{N}\ \ \left(
R(z_{1},...,z_{N})\leftrightarrow \sigma (z_{1},...,z_{N})\right) \}
\end{equation*}%
where $\Sigma $ is any set of axioms defining $\mathcal{V}$. Observe that
since $\mathcal{K}$ is a first order class, it is closed by ultraproducts;
and furthermore, from Lemma \ref{Products Center FHP}, $\mathcal{K}$ is
closed under the formation of direct products. Hence, from Lemma \ref{lema
semantico} we obtain that the following are equivalent:

\begin{enumerate}
\item[(a)] There is an existential positive $\mathcal{L}$-formula which
defines $R$ in $\mathcal{K}$.

\item[(b)] If $(\mathbf{A},R^{\mathbf{A}}),(\mathbf{B},R^{\mathbf{B}})\in 
\mathcal{K}$ and $\sigma :\mathbf{A}\rightarrow \mathbf{B}$ is a
homomorphism, then $\sigma :(\mathbf{A},R^{\mathbf{A}})\rightarrow (\mathbf{B%
},R^{\mathbf{B}})$ is a homomorphism

\item[(c)] There is a $(\exists \bigwedge p=q)$-formula which defines $R$ in 
$\mathcal{K}$.
\end{enumerate}

But $1.$, $2.$ and $3.$ are restatements of $(a)$, $(b)$ and $(c)$,
respectively. That is to say, $1.$, $2.$ and $3.$ are equivalent as required.
\end{proof}

\begin{lemma}
\label{Centrals in FHP} Let $\mathcal{V}$ be a variety with $\vec{0}$ and $%
\vec{1}$ with the FHP and $\mathbf{A}\in \mathcal{V}$. Let $\vec{e}\in A^{N}$%
. Then, $\vec{e}\in Z(\mathbf{A})$ if and only if $\theta ^{\mathbf{A}}(\vec{%
0}^{\mathbf{A}},\vec{e})$ and $\theta ^{\mathbf{A}}(\vec{1}^{\mathbf{A}},%
\vec{e})$ are a pair of complementary factor congruences of $\mathbf{A}$.
\end{lemma}

\begin{proof}
For details of the proof, the reader may consult Corollary 4 of \cite{va7}.
\end{proof}

\begin{lemma}
\label{Lemma Centrals global} Let $\mathcal{V}$ be a variety with $\vec{0}$
and $\vec{1}$ with the FHP and let $\{\mathbf{A}_{i}\}_{i\in I}$ be a family
of non-trivial members of $\mathcal{V}$. Suppose $\mathbf{A}\leq \prod 
\mathbf{A}_{i}$ is a global subdirect product. If $\vec{e}\in A^{N}$ is such
that $\vec{e}(i)\in \{\vec{0}^{\mathbf{A}_{i}},\vec{1}^{\mathbf{A}_{i}}\}$,
for every $i\in I$, then $\vec{e}\in Z(\mathbf{A})$.
\end{lemma}

\begin{proof}
Let $\mathbf{A}\in \mathcal{V}$ and $\vec{e}\in A^{N}$ satisfying the
hypothesis of the statement. Since $\mathcal{V}$ has FHP, in order to prove
that $\vec{e}\in Z(\mathbf{A})$, from Lemma \ref{Centrals in FHP}, we must
verify that $\theta ^{\mathbf{A}}(\vec{0}^{\mathbf{A}},\vec{e})$ and $\theta
^{\mathbf{A}}(\vec{1}^{\mathbf{A}},\vec{e})$ are a pair complementary factor
congruences of $\mathbf{A}$. So, let $J=\{i\in I:\vec{e}(i)=\vec{0}^{\mathbf{%
A}_{i}}\}$. Notice that, since $\mathbf{A}_{i}$ is not trivial, for every $%
i\in I$, then we obtain that $I-J=\{i\in I:\vec{e}(i)=\vec{1}^{\mathbf{A}%
_{i}}\}$. Let $i\in I$. Then, it follows that $\theta _{i}^{\mathbf{A}%
}=\{(a,b)\in A\times A:i\in E(a,b)\}$. Moreover, for every $F\subseteq I$,
it is clear that $\bigcap_{i\in F}\theta _{i}^{\mathbf{A}}=\{(a,b)\in
A\times A:F\subseteq E(a,b)\}$. Now let us to consider $\theta
=\bigcap_{i\in J}\theta _{i}^{\mathbf{A}}$ and $\delta =\bigcap_{i\in
I-J}\theta _{i}^{\mathbf{A}}$. We will prove that $\theta $ and $\delta $
are a pair of complementary factor congruences of $\mathbf{A}$. If $(a,b)\in
\theta \cap \delta $, then $J\subseteq E(a,b)$ and $I-J\subseteq E(a,b)$,
thus $E(a,b)=I$. Hence, $\theta \cap \delta =\Delta ^{\mathbf{A}}$. In order
to show that $\theta \circ \delta =\nabla ^{\mathbf{A}}$, let $(a,b)\in
\nabla ^{\mathbf{A}}$. By assumption, $\mathbf{A}$ is a global subdirect
product of $\{\mathbf{A}_{i}\}_{i\in I}$. So, since $J=%
\bigcap_{k=1}^{N}E(e_{k},0_{k}^{\mathbf{A}})$ and $I-J=%
\bigcap_{k=1}^{N}E(e_{k},1_{k}^{\mathbf{A}})$, then $J$ and $I-J$ are open
sets of the topology over $I$ which contains all the equalizers of elements
of $A$. Therefore, because $\emptyset =J\cap (I-J)\subseteq E(a,b)$, from
the Patchwork Property it follows that there is a $z\in A$, such that $%
J\subseteq E(a,z)$ and $I-J\subseteq E(z,b)$. I.e., $(a,z)\in \theta $ and $%
(z,b)\in \delta $. In consequence, $(a,b)\in \theta \circ \delta $. \newline

\noindent Now we will see that $\theta =\theta ^{\mathbf{A}}(\vec{0}^{%
\mathbf{A}},\vec{e})$. Since $\theta =\{(a,b)\in A\times A:J\subseteq
E(a,b)\}$ and $\vec{e}(i)\in \{\vec{0}^{\mathbf{A}_{i}},\vec{1}^{\mathbf{A}%
_{i}}\}$, then $(0_{k}^{\mathbf{A}},e_{k})\in \theta _{i}^{\mathbf{A}}$ for
every $1\leq k\leq N$, so $\theta ^{\mathbf{A}}(\vec{0}^{\mathbf{A}},\vec{e}%
)\subseteq \theta $. In order to prove the other inclusion, notice that $%
\theta ^{\mathbf{A}}(1_{k}^{\mathbf{A}},e_{k})\subseteq \delta $ for every $%
k $. Therefore $\theta \cap \theta ^{\mathbf{A}}(\vec{1}^{\mathbf{A}},\vec{e}%
)=\Delta ^{\mathbf{A}}$. Because $\mathcal{V}$ has the FHP, thus factor
congruences distribute with any other (c.f.\ {\cite{bi-bu}}) and since $%
\nabla ^{\mathbf{A}}=\theta ^{\mathbf{A}}(\vec{0}^{\mathbf{A}},\vec{e})\vee
\theta ^{\mathbf{A}}(\vec{1}^{\mathbf{A}},\vec{e})$, we obtain that 
\begin{align*}
\theta & =\theta \cap \nabla ^{\mathbf{A}} \\
& =\theta \cap (\theta ^{\mathbf{A}}(\vec{0}^{\mathbf{A}},\vec{e})\vee
\theta ^{\mathbf{A}}(\vec{1}^{\mathbf{A}},\vec{e})) \\
& =(\theta \cap \theta ^{\mathbf{A}}(\vec{0}^{\mathbf{A}},\vec{e}))\vee
(\theta \cap \theta ^{\mathbf{A}}(\vec{1}^{\mathbf{A}},\vec{e})) \\
& =\theta ^{\mathbf{A}}(\vec{0}^{\mathbf{A}},\vec{e})\vee \Delta ^{\mathbf{A}%
} \\
& =\theta ^{\mathbf{A}}(\vec{0}^{\mathbf{A}},\vec{e}).
\end{align*}%
The proof for $\delta =\theta ^{\mathbf{A}}(\vec{1}^{\mathbf{A}},\vec{e})$
is similar. This completes the proof.
\end{proof}

\begin{lemma}
\label{Lemma centrals subalgebras} Let $\mathcal{V}$ be a variety with $\vec{%
0}$ and $\vec{1}$ with the FHP and let $\{\mathbf{B}_{i}\}_{i\in I}$ be a
family of members of $\mathcal{V}$ such that every subalgebra of each $%
\mathbf{B}_{i}$ is directly indecomposable. Suppose $\mathbf{B}\leq \prod 
\mathbf{B}_{i}$ is a global subdirect product, and let $\mathbf{A}\leq 
\mathbf{B}$. Then $Z(\mathbf{A})\subseteq Z(\mathbf{B})$.
\end{lemma}

\begin{proof}
Let us assume that $\vec{e}\in Z(\mathbf{A})$. Then $\vec{e}\in B^{N}$. Let $%
A_{i}=\{a(i):a\in A\}$ and let $\mathbf{A}_{i}$ denotes the algebra whose
universe is $A_{i}$. Since $\mathbf{A}_{i}$ is subalgebra of $\mathbf{B}_{i}$
for every $i\in I$, then it follows that $\mathbf{A}_{i}\in \mathcal{V}_{DI}$%
. Consider the canonical projection $\pi _{i}^{\mathbf{A}}:\mathbf{A}%
\rightarrow \mathbf{A}_{i}$. Because $\pi _{i}^{\mathbf{A}}$ is onto $%
\mathbf{A}_{i}$ and $\mathcal{V}$ has the FHP, from Lemma \ref{In FHP
surjections give homomorphisms} we obtain that $\pi _{i}^{\mathbf{A}}(\vec{e}%
)\in Z(\mathbf{A}_{i})$, for every $\vec{e}\in Z(\mathbf{A})$. But $Z(%
\mathbf{A}_{i})=\{\vec{0}^{\mathbf{A}_{i}},\vec{1}^{\mathbf{A}_{i}}\}$,
therefore $\vec{e}(i)\in \{\vec{0}^{\mathbf{A}_{i}},\vec{1}^{\mathbf{A}%
_{i}}\}=\{\vec{0}^{\mathbf{B}_{i}},\vec{1}^{\mathbf{B}_{i}}\}$. Since $%
\mathbf{B}$ is a global subdirect product of non trivial algebras, from
Lemma \ref{Lemma Centrals global} we conclude that $\vec{e}\in Z(\mathbf{B})$%
.
\end{proof}

\bigskip

Now we can prove the key result.

\begin{theorem}
\label{exitencia del u} Let $\mathcal{L}$ be a language of algebras with at
least a constant symbol. Let $\mathcal{V}$ be a variety of $\mathcal{L}$%
-algebras with the FHP. Suppose that there is a universal class $\mathcal{F}%
\subseteq \mathcal{V}_{DI}$ such that every member of $\mathcal{V}$ is
isomorphic to a global subdirect product with factors in $\mathcal{F}$. Then
there exists a $(M+2)$-ary term $U(x,y,\vec{z})$ and $0$-ary terms $%
0_{1},\ldots ,0_{M},1_{1},\ldots ,1_{M}$ such that%
\begin{equation*}
\mathcal{V}\vDash U(x,y,\vec{0})=x\wedge U(x,y,\vec{1})=y
\end{equation*}
\end{theorem}

\begin{proof}
First notice that, since no algebra of $\mathcal{F}$ has a trivial
subalgebra and every member of $\mathcal{V}$ is a subdirect product with
factors in $\mathcal{F}$, then no non-trivial algebra of $\mathcal{V}$ has a
trivial subalgebra. So, from Proposition 2.3 of \cite{ca-va0} we get that
there are unary terms $0_{1}(w),...,0_{N}(w),1_{1}(w),...,1_{N}(w)$ such that%
\begin{equation*}
\mathcal{V}\vDash \vec{0}(w)=\vec{1}(w)\rightarrow x=y\text{,}
\end{equation*}%
where $w,x,y$ are distinct variables. Let $c\in \mathcal{L}$ be a constant
symbol. Since $\mathcal{V}\vDash \vec{0}(c)=\vec{1}(c)\rightarrow x=y$, we
can redefine $\vec{0}=\vec{0}(c)$ and $\vec{1}=\vec{1}(c)$ to get that $%
\mathcal{V}$ is a variety with $\vec{0}$ and $\vec{1}$. \newline

Next, we will prove that

\begin{enumerate}
\item[(1)] $\mathbb{P}_{u}(\mathcal{V}_{SI})\subseteq \mathcal{V}_{DI}$.
\end{enumerate}

Since every algebra of $\mathcal{V}$ is a subdirect product with factors in $%
\mathcal{F}$, we have that $\mathcal{V}_{SI}\subseteq \mathcal{F}$. So we
have that $\mathbb{P}_{u}(\mathcal{V}_{SI})\subseteq \mathbb{P}_{u}(\mathcal{%
F})\subseteq \mathcal{F}\subseteq \mathcal{V}_{DI}$, which proves (1).

Since by hypothesis we have that every member of $\mathcal{V}$ is isomorphic
to a global subdirect product with factors in $\mathcal{F}$ and $\mathbb{S}(%
\mathcal{F})\subseteq \mathcal{F}\subseteq \mathcal{V}_{DI}$, Lemma \ref%
{Lemma centrals subalgebras} says that

\begin{enumerate}
\item[(2)] If $\mathbf{A}\leq \mathbf{B}\in $ $\mathcal{V}$, then $Z(\mathbf{%
A})\subseteq Z(\mathbf{B})$.
\end{enumerate}

But (2) and Lemma \ref{In FHP surjections give homomorphisms} say that 2. of
Lemma \ref{Preservation centrals} holds, which implies that the property $"%
\vec{e}\in Z(\mathbf{A})"$ is definable in $\mathcal{V}$ by a $(\exists
\bigwedge p=q)$-formula. Let%
\begin{equation*}
\phi (\vec{z})=\exists \vec{w}\bigwedge_{k=1}^{n}p_{k}(\vec{z},\vec{w}%
)=q_{k}(\vec{z},\vec{w})
\end{equation*}%
define the property $``\vec{e}\in Z(\mathbf{A})"$ in $\mathcal{V}$.

Due to in $\mathcal{L}$ there is at least a constant symbol, there exists $%
\mathbf{F}_{\mathcal{V}}\mathbf{(\emptyset )}$. Since $\vec{0},\vec{1}\in Z(%
\mathbf{F}_{\mathcal{V}}\mathbf{(\emptyset )})$, there are $0$-ary terms $%
c_{1},...,c_{N},d_{1},...d_{N}$, such that $\mathbf{F}_{\mathcal{V}}\mathbf{%
(\emptyset )}\models \bigwedge_{k=1}^{n}p_{k}(\vec{0},\vec{c})=q_{k}(\vec{0},%
\vec{c})$ and $\mathbf{F}_{\mathcal{V}}\mathbf{(\emptyset )}\models
\bigwedge_{k=1}^{n}p_{k}(\vec{1},\vec{d})=q_{k}(\vec{1},\vec{d})$. This says
that

\begin{enumerate}
\item[(3)] For every $\mathbf{A}\in \mathcal{V}$,%
\begin{eqnarray*}
p_{k}(\vec{0}^{\mathbf{A}},\vec{c}^{\mathbf{A}}) &=&q_{k}(\vec{0}^{\mathbf{A}%
},\vec{c}^{\mathbf{A}})\text{, }k=1,...,n \\
p_{k}(\vec{1}^{\mathbf{A}},\vec{d}^{\mathbf{A}}) &=&q_{k}(\vec{1}^{\mathbf{A}%
},\vec{d}^{\mathbf{A}})\text{, }k=1,...,n
\end{eqnarray*}
\end{enumerate}

Let $X=\{x,y,\vec{z},\vec{w}\}$ and $\theta =\bigvee_{k=1}^{n}\theta ^{%
\mathbf{\mathbf{F}_{\mathcal{V}}}(X)}(p_{k}(\vec{z},\vec{w}),q_{k}(\vec{z},%
\vec{w}))$. Let $\mathbf{H}=\mathbf{F}_{\mathcal{V}}(X)/\theta $. Since $%
p_{k}^{\mathbf{H}}(\vec{z}/\theta ,\vec{w}/\theta )=q_{k}^{\mathbf{H}}(\vec{z%
}/\theta ,\vec{w}/\theta )$, for $k=1,...,n$, we have that $\mathbf{H}%
\models \phi (\vec{z}/\theta )$. That is to say, $\vec{z}/\theta \in Z(%
\mathbf{H})$. \newline
\newline
\noindent Thereby, since $(x/\theta ,y/\theta )\in \nabla ^{\mathbf{H}%
}=\theta ^{\mathbf{H}}(\vec{z}/\theta ,\vec{0}/\theta )\circ \theta ^{%
\mathbf{H}}(\vec{z}/\theta ,\vec{1}/\theta )$, there is a term $t(x,y,\vec{z}%
,\vec{w})$, such that

\begin{equation*}
(x/\theta ,t^{\mathbf{\mathbf{F}_{\mathcal{V}}}(X)}(x,y,\vec{z},\vec{w}%
)/\theta )\in \theta ^{\mathbf{H}}(\vec{z}/\theta ,\vec{0}/\theta )
\end{equation*}%
and 
\begin{equation*}
(t^{\mathbf{\mathbf{F}_{\mathcal{V}}}(X)}(x,y,\vec{z},\vec{w})/\theta
,y/\theta )\in \theta ^{\mathbf{H}}(\vec{z}/\theta ,\vec{1}/\theta ).
\end{equation*}

\noindent Hence we have

\begin{enumerate}
\item[(4)] $(x/\theta ,t^{\mathbf{H}}(x/\theta ,y/\theta ,\vec{z}/\theta ,%
\vec{w}/\theta ))\in \theta ^{\mathbf{H}}(\vec{z}/\theta ,\vec{0}/\theta )$

\item[(5)] $(t^{\mathbf{H}}(x/\theta ,y/\theta ,\vec{z}/\theta ,\vec{w}%
/\theta ),y/\theta )\in \theta ^{\mathbf{H}}(\vec{z}/\theta ,\vec{1}/\theta
) $
\end{enumerate}

\noindent We will prove that

\begin{enumerate}
\item[(6)] $\mathcal{V}\vDash t(x,y,\vec{0},\vec{c})=x\wedge t(x,y,\vec{1},%
\vec{d})=y$
\end{enumerate}

Let $\mathbf{A}\in \mathcal{V}$ and $a,b\in A$. From (3) and Lemma \ref%
{technical lemma}, there exists a unique $\Omega :\mathbf{H}\rightarrow 
\mathbf{A}$ such that $\Omega (x/\theta )=a$, $\Omega (y/\theta )=b$, $%
\Omega (\vec{z}/\theta )=\vec{0}^{\mathbf{A}}$ and $\Omega (\vec{w}/\theta )=%
\vec{c}^{\mathbf{A}}$. Hence, from (4) and Lemma \ref{homomorfismos
preservan}, we obtain that $t^{\mathbf{A}}(a,b,\vec{0}^{\mathbf{A}},\vec{c}^{%
\mathbf{A}})=a$. In a similar fashion, again from (3) and Lemma \ref%
{technical lemma} applied to (5) together with Lemma \ref{homomorfismos
preservan}, we obtain that $t^{\mathbf{A}}(a,b,\vec{1}^{\mathbf{A}},\vec{d}^{%
\mathbf{A}})=b$ and so (6) is proved.

To conclude the proof we can redefine $\vec{0}%
=(0_{1},...,0_{N},c_{1},...,c_{N})$ and $\vec{1}%
=(1_{1},...,1_{N},d_{1},...,d_{N})$ and take $%
U(x,y,z_{1},...,z_{2N})=t(x,y,z_{1},...,z_{2N})$.
\end{proof}

\bigskip

\section{Pierce stalks}

We shall use the above theorem and some results of \cite{va4} to give
important information on the Pierce stalks for preprimal varieties
corresponding to Rosemberg types 6. and 7. of the introduction. An $h$-ary
relation $\sigma $ on a finite set $P$ is \textit{central} if

\begin{enumerate}
\item[-] it is \textit{totally symmetric}, that is, for all $\vec{a}\in
\sigma $, if $\pi $ is a permutation of $\{1,\ldots ,h\}$, then $(a_{\pi
(1)},\ldots ,a_{\pi (h)})\in \sigma $.

\item[-] it is \textit{totally reflexive}, that is, for all $\vec{a}\in
P^{h} $ with at least two of the $a_{i}$ equal, we have that $\vec{a}\in
\sigma $.

\item[-] there is an $a_{1}$ such that for all $a_{2},\ldots ,a_{h}$ in $P$
we have $\vec{a}\in \sigma $; and

\item[-] $\sigma \neq P^{h}$.
\end{enumerate}

\noindent If $\sigma $ is a central relation on a set $P$, let $\mathbf{P}%
_{\sigma }$ be the preprimal algebra whose universe is $P$ and whose
fundamental operations are all the operations preserving $\sigma $. We use $%
\mathbb{V}(\mathbf{P}_{\sigma })$ to denote the variety generated by $%
\mathbf{P}_{\sigma }$.

\begin{proposition}
Let $\sigma $ be a $h$-ary central relation on $P$, with $h\geq 3$. There is
no universal class $\mathcal{F}\subseteq \mathbb{V}(\mathbf{P}_{\sigma
})_{DI}$ such that every member of $\mathbb{V}(\mathbf{P}_{\sigma })$ is
isomorphic to a global subdirect product with factors in $\mathcal{F}$.
There are Pierce stalks in $\mathbb{V}(\mathbf{P}_{\sigma })$ which are not
directly indecomposable.
\end{proposition}

\begin{proof}
First, we note that by \cite{de} the variety $\mathbb{V}(\mathbf{P}_{\sigma
})$ is congruence distributive, hence it has the FHP (see \cite{FH1970}).
Let $a_{1},\ldots ,a_{N},b_{1},\ldots ,b_{N}\in P$ be such that $%
(a_{1},\ldots ,a_{N})\neq (b_{1},\ldots ,b_{N})$. Let $U:P^{N+2}\rightarrow
P $ be such that%
\begin{equation*}
\begin{array}{ccc}
U(x,y,a_{1},\ldots ,a_{N}) & = & x\medskip \\ 
U(x,y,b_{1},\ldots ,b_{N}) & = & y%
\end{array}%
\end{equation*}%
for any $x,y\in P$. It is easy to check that $U$ does not preserve $\sigma $%
. For example, if $h$ is odd, we can take $(c_{1},\ldots ,c_{h})\notin
\sigma $ and note that%
\begin{equation*}
\begin{array}{rcl}
U(c_{1},c_{2},a_{1},\ldots ,a_{N}) & = & c_{1}\medskip \\ 
U(c_{1},c_{2},b_{1},\ldots ,b_{N}) & = & c_{2}\medskip \\ 
U(c_{3},c_{4},a_{1},\ldots ,a_{N}) & = & c_{3}\medskip \\ 
U(c_{3},c_{4},b_{1},\ldots ,b_{N}) & = & c_{4}\medskip \\ 
\vdots \ \ \ \ \ \ \ \ \ \ \ \ \ \ \ \  &  & \ \vdots \medskip \\ 
U(c_{h-1},c_{h},a_{1},\ldots ,a_{N}) & = & c_{h-1}\medskip \\ 
U(c_{h-1},c_{h},b_{1},\ldots ,b_{N}) & = & c_{h}%
\end{array}%
\end{equation*}%
which, since $\sigma $ is totally reflexive, says that $U$ does not preserve 
$\sigma $. Thus the conclusion of Theorem \ref{exitencia del u} does not
hold and hence there is no universal class $\mathcal{F}\subseteq \mathbb{V}(%
\mathbf{P}_{\sigma })_{DI}$ such that every member of $\mathbb{V}(\mathbf{P}%
_{\sigma })$ is isomorphic to a global subdirect product with factors in $%
\mathcal{F}$. In \cite{bi-bu} it is proved that if every Pierce stalk is in $%
\mathbb{V}(\mathbf{P}_{\sigma })_{DI}$, then the class $\mathbb{V}(\mathbf{P}%
_{\sigma })_{DI}$ is universal. Since the Pierce representation is always a
global subdirect representation, there are Pierce stalks in $\mathbb{V}(%
\mathbf{P}_{\sigma })$ which are not directly indecomposable.
\end{proof}

\bigskip

In order to analyze the case of a $2$-ary central relation we need the
following lemma from \cite{va1}.

\begin{lemma}
\label{weak implica permutable}Let $\mathcal{V}$ be a congruence
distributive variety. If each member of $\mathcal{V}$ is isomorphic to a
global subdirect product with factors in $\mathcal{V}_{SI}$, then $\mathcal{V%
}$ is congruence permutable.
\end{lemma}

\begin{proof}
It follows from (1)$\Rightarrow $(2) of \cite[Corollary. 1 of Thm. 3.4]{va1}.
\end{proof}

\begin{proposition}
Let $\sigma $ be a $2$-ary central relation on a set $P$. Every Pierce stalk
in $\mathbb{V}(\mathbf{P}_{\sigma })$ is directly indecomposable. There are
Pierce stalks in $\mathbb{V}(\mathbf{P}_{\sigma })$ which are not
subdirectly irreducible.
\end{proposition}

\begin{proof}
It is an exercise \cite[X.5.4]{kn} to check that there are term-operations $+
$ and $\times $ on $\mathbf{P}_{\sigma }$, and elements $0,1\in P$ such that 
$\mathbf{P}_{\sigma }$ satisfies the following identities%
\begin{equation*}
\begin{array}{r}
x\times 0=0\times x=0\medskip  \\ 
x\times 1=1\times x=x\medskip  \\ 
x+0=0+x=x\medskip 
\end{array}%
\end{equation*}%
Since $\sigma $ is reflexive, all constant functions on $P$ preserve $\sigma 
$, which says that there are terms $0(w)$ and $1(w)$ such that $0^{\mathbf{P}%
_{\sigma }}(a)=0$ and $1^{\mathbf{P}_{\sigma }}(a)=1$, for every $a\in P$.
Let $U(x,y,z,w)$ be the term $(x\times w)+(y\times z)$. Note that%
\begin{equation*}
\begin{array}{ccc}
U(x,y,0(w),1(w)) & = & x\medskip  \\ 
U(x,y,1(w),0(w)) & = & y%
\end{array}%
\end{equation*}%
are identities of $\mathbb{V}(\mathbf{P}_{\sigma })$. Thus, in the
terminology of \cite{va4}, $\mathbb{V}(\mathbf{P}_{\sigma })$ is a Pierce
variety. Also we note that, since all constant functions on $P$ are
term-operations, we have that $\mathbb{S(}\mathbf{P}_{\sigma })=\{\mathbf{P}%
_{\sigma }\}$. By \cite{de} $\mathbf{P}_{\sigma }$ is the only subdirectly
irreducible algebra in $\mathbb{V}(\mathbf{P}_{\sigma })$. Now, (2)$%
\Rightarrow $(1) of \cite[Theorem 8]{va4} says that every Pierce stalk is
directly indecomposable.

Suppose every Pierce stalk of $\mathbb{V}(\mathbf{P}_{\sigma })$ is
subdirectly irreducible. We will arrive to a contradiction. Since $\mathbf{P}%
_{\sigma }$ is the only subdirectly irreducible algebra in $\mathbb{V}(%
\mathbf{P}_{\sigma })$, we have that $\mathbf{P}_{\sigma }$ is hereditarily
simple. By \cite{de} the variety $\mathbb{V}(\mathbf{P}_{\sigma })$ is
congruence distributive. By Lemma \ref{weak implica permutable} we have that 
$\mathbb{V}(\mathbf{P}_{\sigma })$ is arithmetical. Pixley theorem \cite[%
IV.10.7]{bu-sa} says that $\mathbf{P}_{\sigma }$ is quasiprimal, i.e. the
ternary discriminator is a term-function of $\mathbf{P}_{\sigma }$. It is
easy to check that the ternary discriminator does not preserve $\sigma $.
Thus we have arrived to a contradiction and hence we have proved that there
are Pierce stalks in $\mathbb{V}(\mathbf{P}_{\sigma })$ which are not
subdirectly irreducible.
\end{proof}

\bigskip

Next, we analyze the preprimal variety given by a non-trivial proper
equivalence relation $\sigma $ on a finite set $P$. Let $\mathbf{P}_{\sigma
} $ be the preprimal algebra whose universe is $P$ and whose fundamental
operations are all the operations preserving $\sigma $.

\begin{proposition}
Let $\sigma $ be a non-trivial proper equivalence relation on a finite set $%
P $. Every Pierce stalk in $\mathbb{V}(\mathbf{P}_{\sigma })$ is directly
indecomposable. There are Pierce stalks in $\mathbb{V}(\mathbf{P}_{\sigma })$
which are not subdirectly irreducible.
\end{proposition}

\begin{proof}
Since $\sigma $ is reflexive, we have that the constant operations on $P$
are term-operations of $\mathbf{P}_{\sigma }$. Take $0,1\in P$ such that $%
0/\sigma \neq 1/\sigma $. Let $0(w)$ and $1(w)$ be terms such that $0^{%
\mathbf{P}_{\sigma }}(a)=0$ and $1^{\mathbf{P}_{\sigma }}(a)=1$, for every $%
a\in P$. Define $f:P^{4}\rightarrow P$ as follows%
\begin{equation*}
f(x,y,z,w)=\left\{ 
\begin{array}{ccl}
x &  & \text{if }(z,0)\in \sigma \text{ and }(w,1)\in \sigma \\ 
y &  & \text{if }(z,1)\in \sigma \text{ and }(w,0)\in \sigma \\ 
0 &  & \text{otherwise}%
\end{array}%
\right.
\end{equation*}%
Since $f$ preserves $\sigma $ we have that there is a term $U(x,y,z,w)$ such
that $U^{\mathbf{P}_{\sigma }}=f$. Note that%
\begin{equation*}
\begin{array}{ccc}
U(x,y,0(w),1(w)) & = & x\medskip \\ 
U(x,y,1(w),0(w)) & = & y%
\end{array}%
\end{equation*}%
are identities of $\mathbb{V}(\mathbf{P}_{\sigma })$. Thus, in the
terminology of \cite{va4}, $\mathbb{V}(\mathbf{P}_{\sigma })$ is a Pierce
variety. Note that $\mathrm{Con}(\mathbf{P}_{\sigma })=\{\Delta ^{\mathbf{P}%
_{\sigma }},\sigma ,\nabla ^{\mathbf{P}_{\sigma }}\}$. Since $\mathbb{V}(%
\mathbf{P}_{\sigma })$ is congruence distributive \cite{kn1} and $\mathbb{S(}%
\mathbf{P}_{\sigma })=\{\mathbf{P}_{\sigma }\}$, the subdirectly
irreducibles of $\mathbb{V}(\mathbf{P}_{\sigma })$ are $\mathbf{P}_{\sigma }$
and $\mathbf{P}_{\sigma }/\sigma $. Note that $\mathbf{P}_{\sigma }/\sigma $
is primal. Thus $\mathbb{P}_{u}\mathbb{SV}(\mathbf{P}_{\sigma
})_{SI}\subseteq \mathbb{V}(\mathbf{P}_{\sigma })_{DI}$ and hence (2)$%
\Rightarrow $(1) of \cite[Theorem 8]{va4} says that every Pierce stalk is
directly indecomposable.

Suppose every Pierce stalk of $\mathbb{V}(\mathbf{P}_{\sigma })$ is
subdirectly irreducible. We will arrive to a contradiction. Since every
directly indecomposable algebra of $\mathbb{V}(\mathbf{P}_{\sigma })$ is a
Pierce stalk of itself we have that every directly indecomposable member of $%
\mathbb{V}(\mathbf{P}_{\sigma })$ is subdirectly irreducible. By \cite[%
IV.12.5]{bu-sa} we have that $\mathbb{V}(\mathbf{P}_{\sigma })$ is
semisimple which is impossible since $\sigma \in \mathrm{Con}(\mathbf{P}%
_{\sigma })$. Thus we have arrived to an absurd and hence we have proved
that there are Pierce stalks which are not subdirectly irreducible.
\end{proof}

\end{document}